\documentclass{amsproc}

\usepackage[colorlinks=true, pdfstartview=FitV, linkcolor=blue,
citecolor=blue,
urlcolor=blue,
breaklinks=true]{hyperref}
\usepackage{amsmath,amsfonts,amssymb,amsthm,amscd,comment,euscript,mathrsfs}
\usepackage[usenames]{color}
\usepackage[all]{xy}
\usepackage{paralist}

\newcommand{\arxiv}[1]{\href{http://arxiv.org/abs/#1}{\tt arXiv:\nolinkurl{#1}}}

\newcommand\C{\mathbb{C}}
\newcommand\Z{\mathbb{Z}}

\newcommand\rat{\mathrm{rat}}

\newcommand\Vir{\mathrm{Vir}}
\newcommand\ab{\mathrm{ab}}
\newcommand\rss{\mathrm{ss}}

\newcommand\g{\ensuremath{\mathfrak{g}}}
\newcommand\frM{\mathfrak{M}}
\newcommand\frK{\mathfrak{K}}
\newcommand\sfm{\mathsf{m}}

\newcommand\bx{\mathbf{x}}

\newcommand\cL{\mathcal{L}}

\newcommand\cE{\mathcal{E}}
\newcommand\cR{\mathcal{R}}

\newcommand\ot{\otimes}
\newcommand\lan{\langle} \newcommand\ran{\rangle}
\newcommand\ts{\textstyle}

\newcommand\tta{{\tt a}}
\newcommand\co{\colon}

\renewcommand\AA{\mathbb{A}}

\newcommand\scS{\mathcal{S}}

\newcommand\fra{\mathfrak{a}}

\newcommand\frg{\g}
\newcommand\frh{\mathfrak{h}}
\newcommand\lsl{\ensuremath{\mathfrak{sl}}}

\newcommand\frn{\mathfrak{n}}

\newcommand\frZ{\mathfrak{Z}}
\newcommand\uce{\mathfrak{uce}}

\newcommand{\rmH}{\mathrm{H}}

\newcommand\al{\alpha}

\newcommand\ga{\gamma}
\newcommand\Ga{\Gamma}

\newcommand\la{\lambda}

\newcommand\rh{\rho}
\newcommand\si{\sigma}

\newcommand\ta{\tau}

\DeclareMathOperator{\Hom}{Hom}
\DeclareMathOperator{\Ext}{Ext}

\DeclareMathOperator{\End}{End}
\DeclareMathOperator{\Aut}{Aut}

\DeclareMathOperator{\Spec}{Spec}
\DeclareMathOperator{\Der}{Der}
\DeclareMathOperator{\IDer}{IDer}
\DeclareMathOperator{\Id}{Id}

\DeclareMathOperator{\Supp}{Supp} 
\DeclareMathOperator{\Int}{Int}
\DeclareMathOperator{\Out}{Out}
\DeclareMathOperator{\maxSpec}{maxSpec}

\DeclareMathOperator{\ad}{ad}

\DeclareMathOperator{\ev}{ev}

\DeclareMathOperator{\Ker}{Ker}

\newtheorem{theo}{Theorem}[section]
\newtheorem{prop}[theo]{Proposition}

\newtheorem{cor}[theo]{Corollary}

\theoremstyle{definition}
\newtheorem{defin}[theo]{Definition}
\newtheorem*{rem*}{Remark}

\newtheorem{rems}[theo]{Remarks}
\newtheorem{example}[theo]{Example}

\numberwithin{equation}{section} \allowdisplaybreaks

\renewcommand{\theenumi}{\alph{enumi}}

\makeatletter
\renewcommand{\p@enumii}{\theenumi.}
\makeatother

\newcommand{\lv}[1]{}
\begin{document}
%

\title{A survey of equivariant map algebras with open problems}

\author{Erhard Neher}
\address{Department of Mathematics and Statistics, University of Ottawa, Ottawa, Ontario K1N 6N5, Canada}
\email{neher@uottawa.ca}
\thanks{The first author is supported by a Discovery grant from NSERC}

\author{Alistair Savage}
\address{Department of Mathematics and Statistics, University of Ottawa, Ottawa, Ontario K1N 6N5, Canada}
\email{alistair.savage@uottawa.ca}
\thanks{The second author is supported by a Discovery grant from NSERC}

\date{\today}

\dedicatory{This paper is dedicated to the people of India who made our stay
such a pleasant experience.}


\begin{abstract}
This paper presents an overview of the current state of knowledge in the field of equivariant map algebras
and discusses some open problems in this area.
\end{abstract}

\maketitle \thispagestyle{empty}

\tableofcontents

%
\section*{Introduction}
%

Equivariant map algebras are a large class of algebras that are generalizations of current and loop algebras.  Let $X$ be an affine scheme (e.g.\ an algebraic variety) and let $\fra$ be a finite-dimensional algebra (e.g.\ a Lie or associative algebra), both defined over a field $k$, often assumed to be algebraically closed.  Furthermore, suppose that a finite group $\Gamma$ acts on both $X$ and $\fra$ by automorphisms.  Then the \emph{equivariant map algebra} (or EMA) $\frM = M(X,\g)^\Gamma$ is the algebra of regular maps $X \to \fra$ which are equivariant with respect to the action of $\Gamma$.  One can also give a more algebraic definition of these algebras.  The action of $\Gamma$ on $X$ induces an action on the coordinate ring $A$ of $X$.  Then $\frM = (\fra \otimes A)^\Gamma$, the fixed points of the diagonal action of $\Gamma$ on $\fra \otimes A$.

Perhaps the most important, and certainly the most well-studied, EMAs are the (twisted) loop Lie algebras or, more generally, the (twisted) multiloop Lie algebras.  These are the EMAs where $k$ is an algebraically closed field of characteristic zero, $X$ is the torus $(k^\times)^n$ (where $n=1$ for the loop algebras), $\fra$ is a simple Lie algebra and $\Gamma$ is a product of $n$  cyclic groups acting in a natural way on $X$ and $\fra$ (see~Example~\ref{eg:multiloop}).  Loop Lie algebras play a crucial role in the theory of affine Lie algebras, as do multiloop Lie algebras for extended affine Lie algebras (see \cite{abfp2}).  Other important examples of EMAs include (twisted) current algebras and the Onsager algebra.

However, the class of EMAs is much larger than the set of examples mentioned above.  Nevertheless, it turns out that one can say a significant amount about EMAs and their representation theory in a very general setting.  For instance, for equivariant map Lie or associative algebras, one can classify their irreducible finite-dimensional representations and, in the Lie case, describe the extensions between these representations.

In the current paper we survey the present-day state of knowledge in the theory of equivariant map algebras and their representations.  A significant portion of the exposition is also devoted to open problems in the field. We begin in Section~\ref{sec:general} with the most general definition of an EMA, where the ``target'' algebra is any finite-dimensional algebra (rather than the Lie algebra case most often considered in the literature).  We also introduce there some of the main definitions and examples.  In particular, we discuss the important concept of an evaluation  map.

In Section~\ref{sec:associative}, we consider the case where $\fra$ is an \emph{associative} algebra.  In particular, we give a classification of the irreducible finite-dimensional representations of an equivariant map associative algebra, a result which has not previously appeared in the literature.

In Section~\ref{sec:Lie}, we summarize the most important results concerning equivariant map \emph{Lie} algebras, the most well-studied class of EMAs.  In the case where the target $\fra$ is a finite-dimensional Lie algebra, we recall the classification of the irreducible finite-dimensional representations, the description of the extensions between such representations and the corresponding block decomposition of the category of finite-dimensional representations.  We also review the important concepts of global and local Weyl modules.  We conclude Section~\ref{sec:Lie} with a summary of what is known when the target $\fra$ is not a finite-dimensional Lie algebra.  In particular, we discuss the classification of certain representations in the case that $\fra$ is the Virasoro algebra and the case that $\fra$ is a finite-dimensional basic classical Lie \emph{super}algebra.

Section~\ref{sec:open} is devoted to an overview of open problems in the field of EMAs.  While our list of problems is by no means exhaustive, we have made an attempt to present a wide range of what we consider to be some of the most interesting and important ones.  We begin in Section~\ref{subsec:new-EMAs} with a discussion of the issue of considering \emph{new} types of EMAs.  While the definition of an EMA is quite general, it can be generalized still further, as described there.  In Section~\ref{subsec:prob-arbit} we discuss several important open questions that apply to the general EMA setting.  Finally, in Section~\ref{subsec:open-EMLA} we focus on the equivariant map Lie algebras and describe some of the next steps towards establishing a comprehensive representation theory of these algebras.

\subsection*{Basic notation}

We let $k$ be an algebraically closed field which is of arbitrary characteristic in \S\S\ref{sec:general} and~\ref{sec:associative} and of characteristic zero in the following sections.  We set $k^\times = k \setminus \{0\}$.  Vector spaces and tensor products are over $k$ unless otherwise indicated.

An \emph{algebra} is a vector space $A$ over $k$ together with a $k$-bilinear map $A \times A \to A$, $(a,b) \mapsto ab$, called the \emph{product of $A$}.  For general algebras we write the product by juxtaposition while for Lie algebras we use the traditional notation $[a,b]$.  The direct product of two algebras $A$ and $B$ is denoted $A\boxplus B$ to distinguish it from the direct sum of two vector spaces.  We will use the terms ``module'' and ``representation'' interchangeably.

For schemes, we use the terminology of \cite{EH}.  In particular, an affine scheme $X$ is the (prime) spectrum of a commutative associative unital $k$-algebra $A$. For an arbitrary scheme $X$, we set $A = \mathcal{O}_X(X)$.  Recall (\cite[p.~45]{EH}) that $x \in X$ is a \emph{$k$-rational point of $X$} if its residue field $A/\sfm_x \cong k$ where $\mathsf{m}_x$ is the ideal of $A$ corresponding to $x$.
We let $X_\rat$ denote the set of $k$-rational points of $X$. If $A$ is finitely generated, equivalently $X=\Spec A$ is an affine scheme of finite type, the rational points correspond exactly to the maximal ideals of $A$, i.e., $X_\rat = \maxSpec A$. We say that $X$ is an \emph{affine variety} if $A$ is finitely generated and reduced, in which case we identify $X$ with the maximal spectrum of $A$ (for most readers this will be the most interesting case). For example, a finite-dimensional vector space $V$ is also equipped with the structure of an affine variety, the affine $n$-space $\AA^n$ for $n= \dim_k V$.

The symbol $\Ga$ will always denote a finite group. Although it is not necessary in the beginning, we suppose from the start that the order $|\Ga|$ of $\Ga$ is not divisible by the characteristic of $k$.  (Note that this imposes no condition if $k$ is of characteristic zero.)  If $\Ga$ acts on a set $X$ we let $\Ga_x = \{\ga \in \Ga : \ga \cdot x = x\}$ be the isotropy group of $x\in X$. If $\Ga$ acts on a vector space $V$ we put $V^\Ga = \{v\in V : \ga \cdot v = v \hbox{ for all } \ga \in \Ga \}$.

%
\section{General equivariant map algebras}\label{sec:general}
%

In this section we will define general equivariant map algebras.  By ``general'' we mean ``not necessarily Lie'' (as opposed to the setting of \cite{NSS,NS}).  We work within the following framework:
\begin{enumerate}
  \item $X=\Spec A$ is an affine scheme over $k$.

  \item $\fra$ is a finite-dimensional $k$-algebra, without any further assumptions. We view $\fra$ also as an affine variety.

  \item $\Ga$ is a finite group acting on the scheme $X$, equivalently on the algebra $A$, and on the algebra $\fra$ by automorphisms of the respective structure.
\end{enumerate}

\begin{defin}[Map algebra] \label{Defmap}
  We denote by $M(X,\fra)$ the algebra of regular functions on $X$ with values in $\fra$, called the \emph{untwisted map algebra} or the \emph{algebra of currents}.  Therefore
  \begin{equation} \label{eq:defmap1}
    M(X ,\fra) =\fra \ot A,
  \end{equation}
  with product given by
  \begin{equation} \label{eq:defmap2}
    ({\tt a}_1 \ot a_1) ({\tt a}_2 \ot a_2) = ({\tt a_1a_2}) \ot (a_1 a_2)
  \end{equation}
  for ${\tt a}_i \in \fra$ and $a_i \in A$.  Viewing elements of $M(X,\fra)$ as maps from $X$ to $\fra$, the product of $\al_1, \al_2 \in M(X,\fra)$, is given by $(\al_1 \al_2)(x) = \al_1(x) \al_2(x)$ for $x\in X$.  Although the product $\fra \ot A$ above is an $A$-algebra, we will always view it as a $k$-algebra.
\end{defin}

\begin{rems}
  The multiplication rule \eqref{eq:defmap2} defines of course a $k$-algebra structure on $\fra \ot A$ for an arbitrary $\fra$.  However, with the exception of the map Virasoro algebra discussed in Section~\ref{subsec:super-Virasoro}, we always assume that $\fra$ is finite-dimensional.

  If $\fra$ is an algebra in a class of algebras defined by multilinear identities (a homogeneous variety of algebras in the sense of nonassociative algebras as in \cite{zsss}, e.g.\ Lie algebras or associative algebras), then $M(X,\fra)$ belongs to the same class of algebras.

  If $A=k^n := k \boxplus \cdots \boxplus k$ ($n$ factors), equivalently, $X$ is an algebraic variety consisting of $n$ points, we have $M(X, \fra) \cong \fra^n := \fra \boxplus \cdots \boxplus \fra$ ($n$ factors). If $A= k[t^{\pm 1}]$ is the Laurent polynomial ring, so that $X = k^\times$, the algebra $M(k^\times, \fra)$ is called the \emph{(untwisted) loop algebra of $\fra$}.
\end{rems}

\begin{defin}[Equivariant map algebra]\label{defEMA}
  Observe that $\Ga$ acts on $M(X, \fra)$ by automorphisms: For $\ga\in \Ga$ and $\alpha \in M(X,\fra)$ the map $\ga\cdot \alpha$ is defined by $ (\ga \cdot \alpha)(x) =  \ga \cdot \big( \alpha(\ga^{-1}\cdot x)\big)$ for $x \in X$, i.e., $\ga \cdot (\tta \ot a) = (\ga \cdot \tta ) \ot (\ga \cdot a)$ for $\mathtt{a} \in \fra, a \in A$. We define $M(X,\fra)^\Ga$ to be
  the set of fixed points under this action. That is,
  \[
    \frM :=  M(X,\fra)^\Ga = \{\alpha \in M(X,\fra)  : \alpha(\ga \cdot x) = \ga \cdot   \alpha(x) \, \,  \forall\ x \in X,\ \ga \in \Ga\}
  \]
  is the subalgebra of $M(X,\fra)$ consisting of $\Ga$-equivariant maps from $X$ to $\fra$. We call $M(X,\fra)^\Ga$ an \emph{equivariant map algebra} (or \emph{EMA}).  We note that in general the data $(X,\fra, \Ga)$ are not uniquely determined by $\frM$, see in this respect Section~\ref{prob:isomorp}.
\end{defin}

 As was the case for the map algebras, it is clear that $\frM$ inherits any multilinear identities satisfied by $\fra$. For example, $\frM$ is associative if $\fra$ is associative, etc.\ (we trust the reader will fill in here his or her favorite class of algebras).

Let $\Xi$ be the set of isomorphism classes of irreducible representations of $\Ga$. Since $\Ga$ acts completely reducibly on $\fra$ and $A$, each algebra has a unique decomposition into isotypic components.  Thus we have $\fra = \bigoplus_{\xi \in \Xi} \fra_\xi$ where the isotypic component $\fra_\xi$ is the sum of all $\Ga$-submodules of $\fra$ of type $\xi$.  Similarly $A= \bigoplus_{\xi \in \Xi} A_\xi$. We denote by $0$ the class of the trivial representation and by $-\xi $ the equivalence class of the representation dual to the representation in $\xi$. Since $(\fra_\xi \ot A_\ta)^\Ga = \{0\}$ unless $\tau = - \xi$ we get
\begin{equation} \label{eq:frm-general}
  \frM  = \ts \bigoplus_{\xi \in \Xi} (\fra_\xi \ot A_{-\xi})^\Ga \supseteq \fra_0 \ot A_0.
\end{equation}
One of the difficulties in studying equivariant map algebras is that it is in general difficult to understand the elements in $(\fra_\xi  \ot A_{-\xi})^\Ga$. The situation is (somewhat) easier in the case where $\Ga$ is abelian, see Example~\ref{eg:abelian}.

Equivariant map algebras were introduced in \cite{NSS} for Lie algebras $\fra$, but it was mentioned there that some of the preliminary results of \cite{NSS} hold for arbitrary $\fra$. We will come back to this below. We first discuss some examples so that the reader can appreciate the full scope of the definition.  We focus on examples where $\Ga$ is abelian. The interested reader can find an example with a non-abelian $\Ga$ in \cite[Ex.~3.13]{NSS}.

\begin{example}[$\Ga$ abelian]\label{eg:abelian}
  If $\Ga$ is abelian, the isomorphism classes of irreducible representations can be identified with the character group of $\Ga$ which, for simplicity, we also denote by $\Xi$. This is an abelian group (non-canonically isomorphic to $\Ga$), whose group operation we will write additively. The isotypic component $A_\xi $ is given by $A_\xi = \{a\in A: \ga \cdot a = \xi(\ga) a \hbox{ for all } \ga \in \Ga\}$, whence $(\fra_\xi \ot A_{-\xi})^\Ga = \fra_\xi \ot A_{-\xi}$. The decomposition~\eqref{eq:frm-general} now reads
  \begin{equation} \label{eq:xrat0}
    \frM = \ts \bigoplus_{\xi \in \Xi} \, \fra_\xi \ot A_{-\xi}.
  \end{equation}
  Note that the decomposition \eqref{eq:xrat0} is a $\Xi$-grading of the algebra $\frM$. The next example introduces an important special case.
\end{example}

\begin{example}[Multiloop algebras] \label{eg:multiloop}
  Fix positive integers $n, m_1, \dots, m_n$.  Let
  \[
    \Gamma = \langle \gamma_1,\dots, \gamma_n : \gamma_i^{m_i}=1,\ \gamma_i \gamma_j = \gamma_j \gamma_i,\ \forall\ 1 \le i,j \le n \rangle \cong (\Z/m_1\Z) \times \dots \times (\Z/m_n\Z)
  \]
  and suppose that $\Gamma$ acts on $\fra$ by automorphisms. Note that this is equivalent to specifying commuting automorphisms $\sigma_i$, $i=1,\dots,n$, of $\fra$ such that $\sigma_i^{m_i}=\Id$. For $i = 1,\dots, n$, let $\xi_i$ be a primitive $m_i$-th root of unity. Let $X=(k^\times)^n$ and define an
  action of $\Gamma$ on $X$ by
  \[
    \gamma_i \cdot (z_1, \dots, z_n) = (z_1, \dots, z_{i-1}, \xi_i z_i, z_{i+1}, \dots, z_n).
  \]
  Then
  \begin{equation} \label{eq:twisted-multiloop-def}
      M(\fra,\sigma_1,\dots,\sigma_n,m_1,\dots,m_n) := M(X,\fra)^\Gamma
  \end{equation}
  is the \emph{multiloop algebra} of $\fra$ relative to $(\sigma_1, \dots, \sigma_n)$ and $(m_1, \ldots, m_n)$. General multiloop algebras are studied in detail in \cite{abp2} and \cite{abp2.5}.
\end{example}

\begin{defin}[Evaluation map]\label{def:evalmap}
  For $x\in X_\rat$ with isotropy group $\Ga_x$ we set
  \[
    \fra^x = \{ \tta \in \fra : \ga \cdot \tta = \tta \hbox{ for all } \ga \in \Ga_x\}
  \]
  and note that $\fra^x$ is, in general, a proper subalgebra of $\fra$. We denote by $X_*$ the set of finite subsets $\bx \subseteq X_\rat$ for which $(\Ga \cdot x) \cap (\Ga \cdot x') = \emptyset$ for distinct $x,x'\in \bx$.  For $\bx \in X_*$ we define $\fra^\bx = \boxplus_{x\in \bx} \fra^x$. The
  \emph{evaluation map}
  \[
    \ev_\bx \colon \frM \to \fra^\bx , \quad \ev_\bx(\al)= \big( \al(x) \big)_{x \in \bx}
  \]
  is a surjective algebra homomorphism (see~\cite[Cor.~4.6]{NSS}).  This has important consequences for the representation theory of $\frM$.  When $\bx = \{x\}$ is a singleton, we will often denote the evaluation map by $\ev_x$.
\end{defin}

For the next theorem we recall that an ideal of an arbitrary $k$-algebra $B$ is a subspace $I$ of $B$ such that $xb \in I$ and $bx \in I$ for all $x\in I$ and $b\in B$. One calls $B$ \emph{simple} if $B$ is not the zero algebra and if every ideal $I$ of $B$ is trivial: $I=\{0\}$ or $I=B$.

\begin{prop} \label{prop:fact}
  Let $\frK$ be an ideal of $\frM$ such that the quotient algebra $\frM/\frK$ is finite-dimensional and simple. Then there exists a $k$-rational point $x\in X$ such that the canonical epimorphism $\pi : \frM \to \frM/\frK$ factors through the evaluation map $\ev_x : \frM \to \fra^x:$
  \[
    \xymatrix{
      \frM \ar@{->>}[rr]^\pi \ar@{->>}[rd]_{\ev_x} && \frM/\frK \\ & \fra^x \ar@{->>}[ru]
    }
  \]
\end{prop}

\begin{proof}
  This is Proposition~5.2 together with Remark~5.3 in \cite{NSS}.
\end{proof}

%
\section{Equivariant map associative algebras} \label{sec:associative}
%

In this section, we assume that $\fra$ is a unital associative algebra and use Proposition~\ref{prop:fact} to obtain a new description of representations of equivariant map algebras $M(X,\fra)^\Ga$.  As pointed out in Section~\ref{sec:general}, in this case $M(X,\fra)^\Ga$ is an associative algebra, so that it makes sense to speak of representations. We recall that a representation of an associative algebra $B$ is an algebra homomorphism $\rho \co B \to \End_k V$ where $\End_k V$ denotes the associative $k$-algebra of endomorphisms of a $k$-vector space $V$.

\begin{theo} \label{th:rep-ass}
  Let $\rho : M(X,\fra)^\Ga \to \End_k V$ be an irreducible finite-dimensional representation of\/  $\frM= M(X,\fra)^\Ga$. Then there exists a $k$-rational point $x\in X$ such that $\rho$ factors as
  \[
    \xymatrix@1{
      \frM \ar@{->>}[r]^{\ev_x} & \fra^x \ar@{->>}[r]^>>>>>{\rho_x} & \End_k V
    }
  \]
  where $\rho_x$ is an irreducible finite-dimensional representation of the finite-dimensional algebra $\fra^x$.

  Conversely, for any $x\in X_\rat$ and any irreducible representation $\rh_x \colon \fra^x \to \End_k V$ the map $\rh_x \circ \ev_x$ is an irreducible representation of $\frM$.
\end{theo}

\begin{proof}
  Let $\frK = \Ker \rho$. Then $\overline{\frM} = \frM/\frK$ is a finite-dimensional associative algebra with a faithful irreducible representation $\bar \rho$ satisfying $\rho = \bar \rho \circ \pi$ for $\pi \colon \frM \to \frM /\frK$ the canonical projection. By Burnside's Theorem (see, for
  example, \cite[\S4.5]{jake:BAII}), $\bar \rho$ is an isomorphism, in particular $\overline{\frM}$ is simple.   Thus, by Proposition~\ref{prop:fact}, there exists a rational point $x$ such that $\rho$ factors through $\ev_x$.

  That, conversely, $\rho_x \circ \ev_x$ is an irreducible representation of $\frM$ follows from the surjectivity of $\ev_x$.
\end{proof}

We call the representations $\rho_x \circ \ev_x = \rho_x \ev_x$ appearing in Theorem~\ref{th:rep-ass} \emph{single point evaluation representations}. The theorem can then be summarized by saying that \emph{the irreducible finite-dimensional representations of $M(X,\fra)^\Ga$ are precisely the  single point evaluation representations.}

We now turn to describing the isomorphism classes of these representations.  We will use an approach that will reappear in the setting of Lie algebras.

For $x\in X_\rat$, let $\mathcal{R}_x$ denote the set of isomorphism classes of irreducible finite-dimensional representations of $\fra^x$, and put $\mathcal{R}_X=\bigsqcup_{x \in X_\rat} \mathcal{R}_x$ (disjoint union).  It is convenient to view the trivial zero-dimensional representation as an irreducible representation; we will use $0$ to denote its class in $\mathcal{R}_x$.  Note that $\Gamma$ acts on $\mathcal{R}_X$ by
\[
  \Gamma \times \mathcal{R}_X \to \mathcal{R}_X,\quad (\gamma,[\rho]) \mapsto \gamma  \cdot [\rho] := [\rho \circ \gamma^{-1}] \in \mathcal{R}_{\gamma \cdot x},
\]
where $[\rho] \in \mathcal{R}_x$ is the isomorphism class of a representation $\rho$ of $\fra^x$. For isomorphic irreducible representations $\rho$ and $\rho'$ of $\fra^x$, the evaluation representations $\ev_x \rho$ and $\ev_x \rho'$ are isomorphic representations of $\frM$. Therefore, for $[\rho] \in \mathcal{R}_x$, we can define $\ev_x [\rho]$ to be the isomorphism class of $\ev_x \rho$, and this is independent of the representative $\rho$.

We let $\cE^\Gamma_\mathrm{single}$ denote the set of $\Gamma$-equivariant functions $\psi : X_\rat \to \mathcal{R}_X$ such that $\psi(x) \in \mathcal{R}_x$ for all $x \in X_\rat$ and $\Supp \psi := \{x \in X_\rat: \psi(x) \ne 0\}$ is contained in a single $\Ga$-orbit.  For $0\ne \psi \in \cE^\Gamma_\mathrm{single}$ we define $\ev^\Gamma_\psi = \ev_x [\psi(x)]$ where $\Ga \cdot \{x\} = \Supp \psi$. One can show as in \cite[Lem.~4.13]{NSS}, that $\ev^\Gamma_\psi$ is independent of the choice of $x\in \Supp \psi$. We let $\ev^\Gamma_\psi$ be the zero-dimensional representation if $\psi = 0$.  Thus $\psi \mapsto \ev^\Gamma_\psi$ defines a map $\ev^\Gamma \co \cE^\Gamma_\mathrm{single} \to \scS$, where $\scS$ denotes the set of isomorphism classes of irreducible finite-dimensional representations of $\frM$.

\begin{cor} \label{cor:assoc-enumeration}
  The map $\ev^\Gamma\co \cE^\Gamma_\mathrm{single} \to \scS$ is a bijection.  In other words, $\cE^\Gamma_\mathrm{single}$ enumerates the isomorphism classes of irreducible finite-dimensional representations of $\frM$.
\end{cor}

\begin{proof}
  Surjectivity follows from Theorem~\ref{th:rep-ass} and injectivity can be be shown as in \cite[Prop.~4.15]{NSS}.
\end{proof}

Note that in the associative case, all irreducible finite-dimensional representations are \emph{single} orbit evaluation representations, in contrast to the Lie algebra case (see Theorem~\ref{thm:Lie}).  This is not surprising since it is not clear how to define multi-orbit evaluation representations in the associative case unless one has a bialgebra structure on the target $\fra$.

%
\section{Equivariant map Lie algebras} \label{sec:Lie}
%

In this section we consider the most well studied case of equivariant map algebras: the case where the target algebra is a Lie algebra.  Unless otherwise stated, we assume throughout this section that $\g$ is a finite-dimensional Lie algebra and that $A$ is a commutative associative algebra (not necessarily finitely generated) over an algebraically closed field $k$ of characteristic zero.

\subsection{Irreducible finite-dimensional representations} \label{subsec:Lie-irr-fd-reps}

Define $\cR_x$, $x \in X_\rat$, and $\cR_X$ as in Section~\ref{sec:associative} (with $\fra$ replaced by $\g$), except that $0 \in \cR_x$ denotes the trivial one-dimensional representation and $\ev_\psi$ is the trivial one-dimensional representation if $\psi=0$.  Let $\cE^\Gamma$ denote the set of $\Gamma$-equivariant functions $\psi : X_\rat \to \cR_X$ such that $\psi(x) \in \cR_x$ for all $x \in X_\rat$ and $\Supp \psi$ is finite.

\begin{defin}[{Evaluation representation \cite[Def.~4.14]{NSS}}] \label{def:eval-rep}
  For $\psi \in \cE^\Gamma$, we define $\ev^\Gamma_\psi = \ev_\bx (\psi(x))_{x \in \bx}$, where $\bx \in X_*$ contains one element of each $\Gamma$-orbit in $\Supp \psi$.  Namely, if $\rho_x \co \g^x \to \End_k V_x$, $x \in \bx$, is an irreducible representation with isomorphism class $[\psi(x)]$ then $\ev_\psi^\Gamma$ is the composition
  \[ \textstyle
    M(X,\g)^\Gamma \xrightarrow{\ev_\bx} \boxplus_{x \in \bx}
    \, \g^x \xrightarrow{\otimes_{x \in \bx} \rho_x}
    \End_k (\bigotimes_{x \in \bx} V_x).
  \]
  We call $\ev^\Ga_\psi$ an \emph{evaluation representation} and note that it is always an irreducible finite-dimensional representation of $\frM=M(X,\g)^\Ga$ (see \cite[Prop.~4.9]{NSS}).\footnote{In \cite{NSS} an evaluation representation need not necessarily be irreducible. We have chosen to include irreducibility in the definition of an evaluation representation here since we will not deal with non-irreducible evaluation representations in this paper.}  By \cite[Lem.~4.13]{NSS}, $\ev^\Gamma_\psi$ is independent of the choice of $\bx$.  If $\psi$ is the map that is identically 0 on $X$, we define $\ev^\Gamma_\psi$ to be the isomorphism class of the trivial one-dimensional representation of $\frM$.
   Thus $\psi \mapsto \ev^\Gamma_\psi$ defines a map $\cE^\Gamma \to \scS$, where $\scS$ denotes the set of isomorphism classes of irreducible finite-dimensional representations of $\frM$.
\end{defin}

In the classification of irreducible finite-dimensional representations of equivariant map algebras, there are two essential differences between the associative and Lie cases.  First, one needs to consider evaluation representations supported on more than one orbit (compare $\cE^\Gamma$ to the $\cE^\Gamma_\mathrm{single}$ appearing in Corollary~\ref{cor:assoc-enumeration}).  Second, in the Lie case one needs, in general, to consider one-dimensional representations since $\frM$ need not be perfect.

Recall that the isomorphism classes of one-dimensional representations of a Lie algebra $L$ are naturally identified with the elements of the space $(L/[L,L])^*$.

\begin{theo}[{\cite[Thm.~5.5]{NSS}}] \label{thm:Lie}
  The map
  \[
    (\frM/[\frM,\frM])^* \times \cE^\Gamma \to \scS,\quad (\lambda, \psi) \mapsto \lambda \otimes \ev^\Gamma_\psi,\qquad \lambda \in (\frM/[\frM,\frM])^*,\quad \psi \in \cE^\Gamma,
  \]
  is surjective.  In particular, all irreducible finite-dimensional representations of $\frM$ are tensor products of an evaluation representation and a one-dimensional representation. If $\frM$ is perfect (i.e., $[\frM,\frM]=\frM$), then the map
  \[
    \cE^\Gamma \to \scS,\quad \psi \mapsto \ev^\Gamma_\psi,
  \]
  is a bijection.
\end{theo}

In fact, Theorem~\ref{thm:Lie} can be made more precise.  We have that $\lambda \otimes \ev^\Gamma_\psi = \lambda' \otimes \ev^\Gamma_{\psi'}$ if and only if there exists $\Phi \in \cE^\Gamma$ such that $\dim \ev^\Gamma_\Phi=1$, $\lambda' = \lambda - \ev^\Gamma_\Phi$ and $\ev^\Gamma_{\psi'} = \ev^\Gamma_{\psi \otimes \Phi}$.  Here $\psi \otimes \Phi \in \cE^\Gamma$ is given by $(\psi \otimes \Phi)(x) = \psi(x) \otimes \Phi(x)$, where $\psi(x)$ (respectively $\Phi(x)$) is the one-dimensional trivial representation if $x \not \in \Supp \psi$ (respectively $x \not \in \Supp \Phi$). In particular, the restriction of the map $(\lambda, \psi) \mapsto \lambda \otimes \ev^\Gamma_{\psi}$ to either factor (times the zero element of the other) is injective.

We see that Theorem~\ref{thm:Lie} is simplified when $\frM$ is perfect.  This happens, for example, if $[\g^\Gamma,\g]=\g$ or if $\g$ is semisimple and $\Gamma$ acts on $\g$ by diagram automorphisms (see \cite[Cor.~5.8]{NSS}) or if $\frM$ is a multiloop algebra (see \cite[Lem.~4.9]{abp2.5}).

Theorem~\ref{thm:Lie} is a generalization of the work of many authors.  The classification of irreducible finite-dimensional representations of (twisted) loop algebras (i.e.\ when $A=\C[t,t^{-1}]$) dates back to the work of Chari and Pressley (see \cite{Cha86,CP86,CP87,CP98}).  For the multiloop algebras, the irreducible modules were described in the untwisted case by Rao in \cite{ER01}.  In certain twisted cases, they were then described by Batra in \cite{Bat04}.  A complete classification of the irreducible finite-dimensional modules for twisted multiloop algebras was given by Lau in \cite{Lau10}.  In the setting that $\g$ is a simple Lie algebra, the action of $\Gamma$ is trivial and $A$ is finitely generated, the irreducible finite-dimensional representations were classified by Chari, Fourier and Khandai in \cite{CFK10} using the theory of Weyl modules.

\subsection{Extensions and block decompositions}

The category of finite-dimen\-sional representations  of equivariant map Lie algebras is, in general, not semisimple.  For example, the local Weyl modules (see Definition~\ref{def:local-Weyl}) are indecomposable but in general not irreducible.  It is therefore important to know the extensions between irreducible objects in this category.  In this section we review what is known regarding such extensions and the resulting block decomposition of the category of finite-dimensional representations.  We focus on the issue of extensions between evaluation representations, referring the reader to \cite{NS} for details on the more general case.

We assume throughout this section that $\g$ is a finite-dimensional reductive Lie algebra and that $A$ is finitely generated.  We note that in this case all subalgebras $\g^x$, $x\in X_\rat$, are also reductive.  For a Lie algebra $L$, we set $L_\ab := L/[L,L]$.  If $L$ is reductive, we let $L_\rss := [L,L]$ be its semisimple part.

The following result reduces the determination of extensions between evaluation modules to the computation of extensions between evaluation modules supported on a single orbit.  In the special case where $\Gamma$ is trivial and $\g$ is semisimple, it was first proven by Kodera (\cite[Thm.~3.6]{Kod10}).

\begin{theo}[{\cite[Thm.~3.7]{NS}}] \label{theo:two-eval}
  Suppose $V$ and $V'$ are evaluation modules corresponding to $\psi, \psi' \in \cE^\Gamma$ respectively.  Let $V = \bigotimes_{x \in \bx} V_x$ and $V' = \bigotimes_{x \in \bx} V_x'$ for some $\bx \in X_*$, where $V_x, V_x'$ are (possibly trivial) evaluation modules at the point $x \in \bx$.  Then the following are true.
  \begin{enumerate}
    \item \label{theo-item:two-eval:multiple-orbit-diff}
      If $\psi$ and $\psi'$ differ on more than one $\Gamma$-orbit, then $\Ext^1_\frM(V,V')=0$.

    \item \label{theo-item:two-eval:one-orbit-diff}
      If $\psi$ and $\psi'$ differ on exactly one orbit $\Gamma \cdot x_0$, $x_0 \in \bx$, then
      \[
        \Ext^1_\frM(V,V') \cong \Ext^1_\frM(V_{x_0},V'_{x_0}).
      \]

    \item \label{theo-item:two-eval:isom} If $\psi = \psi'$ (so $V \cong V')$, then
      \[ \textstyle
        (\frM_\ab^*)^{|\bx| -1 } \oplus \Ext^1_\frM (V, V) \cong \bigoplus_{x\in \bx} \Ext^1_\frM (V_x, V_x).
      \]
  \end{enumerate}
\end{theo}

The next result allows one to compute the extensions  between two evaluation modules supported on the same (single) orbit, which is essential in view of Theorem~\ref{theo:two-eval}.

\begin{theo}[{\cite[Thm.~3.9]{NS}}] \label{theo:ext-eval-point}
  Let $V$ and $V'$ be two single orbit evaluation representations supported on the same orbit $\Gamma \cdot x$ for some $x \in X_\rat$. Thus ${\g^x_\ab}$ acts on $V$ and $V'$ by linear forms, which we denote by $\la$ and $\la'$ respectively. Let
  \begin{gather*}
    \frK_x := \ker \ev_x, \\
    \frZ_x := \ev^{-1}_x({\g^x_\ab}) = \{ \al \in \frM : [\al, \frM] \subseteq \frK_x \}.
  \end{gather*}
  Then
  \begin{equation} \label{eq:evap1}
    \Ext_\frM^1(V, V') \cong
    \begin{cases}
      \Hom_{\g^x} (\frK_{x,\ab}, V^* \ot V') & \hbox{if $\la \ne \la'$}, \\
      \Hom_{\g^x_\rss}\,( \frZ_{x,\ab}, V^*\ot V') & \hbox{if $\la= \la'$}.
    \end{cases}
  \end{equation}
  In particular, if $\g^x$ is semisimple, then $\g^x=\g^x_\rss$, $\la=\la'=0$,
  $\frK_x=\frZ_x$, and
  \begin{equation} \label{eq:evap2}
    \Ext_\frM^1(V, V') \cong
    \Hom_{\g^x} (\frK_{x,\ab}, V^* \ot V').
  \end{equation}
\end{theo}

 In the special case that $\Gamma$ is abelian and acts freely on $X_\rat$, the extensions are very simple to describe as the following result shows.

\begin{prop}[{\cite[Prop.~4.9]{NS}}] \label{prop:ext-free-abelian}
  Suppose $\Gamma$ is abelian, $\Gamma$ acts freely on $X_\rat$ and $\g$ is semisimple.  Then for any two evaluation modules $V_1$, $V_2$ at $x$ we have
  \[
    \Ext^1_\frM(V_1, V_2) \cong \Hom_\g( \g, V_1^* \ot V_2) \ot (\sfm_x/\sfm_x^2)^*,
  \]
  where we recall that $\sfm_x$ is the maximal ideal of $A$ corresponding to $x$.
\end{prop}

Note that $(\sfm_x/\sfm_x^2)^*$ is the Zariski tangent space of $X$ at the point $x$. Proposition~\ref{prop:ext-free-abelian} was first proved in the case where $\Gamma$ is trivial by Kodera (\cite[Prop.~3.1]{Kod10}).  In the case of the current algebra (where $k=\C$ and $A=\C[t]$), extensions were also described in \cite{CG05}.

Having an explicit formula for the extensions between simple  objects in a category allows one (at least in theory) to compute the block decomposition of that category. In the case that $\g$ is semisimple and $\Gamma$ is abelian and acts freely on $X_\rat$, the description of the blocks is quite simple.  First note that $\Gamma$ acts on the weight lattice $P$ of $\g$ via the quotient $\Aut \g \twoheadrightarrow \Aut \g/ \Int \g \cong \Out \g$, where $\Int \g$ denotes the group of inner automorphisms of $\g$ and $\Out \g$ denotes the group of diagram automorphisms of $\g$.

\begin{prop}[{\cite[Prop.~6.6]{NS}}] \label{prop:block-decomp}
  Suppose that $\Gamma$ is abelian and acts freely on $X_\rat$ and that $\g$ is semisimple.  Furthermore, assume that, for all $x \in X_\rat$, the tangent space $(\sfm_x/\sfm_x^2)^* \ne 0$.  For example, assume that $X$ is an irreducible algebraic variety of positive dimension.  Then the blocks of the category of finite-dimensional representations of $\frM$ are naturally enumerated by finitely supported equivariant maps $X_\rat \to P/Q$.
\end{prop}

Special cases of EMAs satisfying the above assumptions are loop algebras, for which block decompositions were described in \cite{CM04} for the untwisted case and in \cite{Sen10} for the twisted case.  In the case that $\Gamma$ is trivial, the block decomposition was described in \cite[\S4.2]{Kod10}.  We refer the reader to \cite[\S5]{NS} for a more general description (i.e.\ with weaker assumptions than those of Proposition~\ref{prop:block-decomp}).

Combining the above with the results of Section~\ref{subsec:Lie-irr-fd-reps}, one sees a particularly nice pattern emerge in the theory of EMAs and their representations.  Namely, many important quantities are parameterized by sets of equivariant maps on $X$ (or $X_\rat$):
\begin{itemize}
  \item The EMA $\frM$ itself is the space of equivariant maps $X \to \g$.
  \item Provided $\frM$ is perfect (which is the case in many important examples), the irreducible finite-dimensional representations are parameterized by the set of finitely supported equivariant maps from $X_\rat$ to the set of isomorphism classes of irreducible finite-dimensional representations of the subalgebras $\g^x$, $x\in X_\rat$.
  \item Provided $\Gamma$ is abelian and acts freely on $X_\rat$, $\g$ is semisimple and $A$ finitely generated, the blocks of the category of finite-dimensional representations are parameterized by the set of finitely supported equivariant maps $X_\rat \to P/Q$.
\end{itemize}

\subsection{Weyl modules}

In this section we present an overview of the present state of the theory of Weyl modules.  We
begin with the untwisted case (i.e.\ when $\Gamma$ is trivial).  We assume
throughout this section that $\g$ is a simple Lie algebra.  As before,  $A$ is a commutative associative algebra with unit and both $\g$ and $A$ are defined over an algebraically closed field $k$ of characteristic zero.  We fix a triangular decomposition $\g = \frn^- \oplus \frh \oplus \frn^+$ of $\g$ and let $P^+$ and $Q^+$ denote the dominant integral weight lattice and positive root lattice of $\g$ respectively.  We also identify $\g$ with the subalgebra $\g \otimes k \subseteq \g \otimes A$

Given a left $\g$-module $V$, let
\[
  P_A(V) = U(\g \otimes A) \otimes_{U(\g)} V.
\]
Now suppose $\lambda \in P^+$ and  let $V(\lambda)$ be the corresponding
irreducible highest weight module of $\g$.

\begin{defin}[Untwisted global Weyl module]
  The \emph{(untwisted) global Weyl module} $W_A(\lambda)$ associated to $\lambda \in P^+$ is the unique maximal quotient (of $(\g \otimes A)$-modules) of $P_A(V(\lambda))$ whose weights are contained in $\lambda - Q^+$.  In other words
  \[ \ts
    W_A(\lambda) := P_A(V(\lambda)) \Big/ \sum_{\mu \not \le \lambda} U(\g \otimes A)P_A(V(\lambda))_\mu.
  \]
\end{defin}

It is possible to give an explicit description of the global Weyl modules in
terms of generators and relations.  For this, we fix a set $\{e_i,f_i,h_i\}_{i \in I}$ of Chevalley generators of $\g$.

\begin{prop}[{\cite[Prop.~4]{CFK10}}]
  For $\lambda \in P^+$, the global Weyl module $W_A(\lambda)$ is generated as a $(\g \otimes A)$-module by a single vector $w_\lambda$ with defining relations
  \[
    (\frn^+ \otimes A)w_\lambda = 0,\quad hw_\lambda = \lambda(h)w_\lambda,\quad f_i^{\lambda(h_i)+1}w_\lambda = 0,\qquad i \in I,\quad h \in \frh.
  \]
\end{prop}

We now define the local Weyl modules.  We define $\mathcal{R}_x$ and $\mathcal{R}_X$ as in Section~\ref{subsec:Lie-irr-fd-reps}.  For $x \in X_\rat$, we identify $\cR_x$ with $P^+$ by associating to an isomorphism class of irreducible finite-dimensional $\g$-modules the highest weight of the modules in that class.  In this way we identify $\cE$ with the set of finitely supported maps from $X_\rat$ to $P^+$.

\begin{defin}[Untwisted local Weyl module] \label{def:local-Weyl}
  Given $\psi \in \cE$ (see Section~\ref{subsec:Lie-irr-fd-reps}), the \emph{(untwisted) local Weyl module} $W(\psi)$ is the $(\g \otimes A)$-module generated by a nonzero vector $w_\psi$ with defining relations
  \begin{gather*}
    (\frn^+ \otimes A)w_\psi=0, \\
    \ts \alpha w_\psi = \left( \sum_{x \in \Supp \psi} \psi(x)(\alpha(x)) \right) w_\psi,\quad \alpha \in \frh \otimes A, \\ \ts
    (f_i \otimes 1)^{\lambda(h_i)+1}w_\psi=0,\quad i \in I,\quad \text{where } \lambda = \sum_{x \in \Supp \psi} \psi(x).
\end{gather*}
\end{defin}

Global and local Weyl modules were first defined for affine algebras (i.e.\ loop algebras, where $X=k^\times$) by Chari and Pressley in \cite{CP01}, where many of their important properties were proved.  This definition and some analogous results were extended to the case where $A$ is the coordinate ring of an algebraic variety by Feigin and Loktev in \cite{FL04}.  A more general functorial approach to Weyl modules for arbitrary finitely generated $A$ was taken by Chari, Fourier and Khandai in \cite{CFK10}.

In the loop case, where $A=\C[t,t^{-1}]$, local Weyl modules are  $q=1$ limits of irreducible representations of quantum affine algebras.  This was conjectured in \cite{CP01} and proved there in the case $\g = \lsl_2$.  It was also shown there that the general case follows from conjectural dimension formulas for these modules.  Using these results, the statement for general $\g$ then follows from results in \cite{CL06,FoL07,Nao10}.  However, the local Weyl modules do not remain
irreducible in the $q \to 1$ limit.  In general, local Weyl modules are reducible indecomposable representations.  The module $W(\psi)$ is the largest indecomposable representation of highest weight $\lambda$ (as a $\g$-module) and irreducible quotient $\ev^\Gamma_\psi$.

It is natural to seek to extend the notion of local and global Weyl modules to the twisted setting, where $\Gamma$ is nontrivial.  Some progress has been made towards this goal (see Section~\ref{prob:Lie-rep}\eqref{open-prob:Weyl-modules} for some open problems in this direction).  In \cite{CFS08}, local Weyl modules were defined and studied for the twisted loop algebras.  The main idea there was to consider the restrictions of local Weyl modules for loop algebras to the twisted subalgebras.  More generally, in \cite{FKKS12}, local Weyl modules were defined and studied in the case that $\Gamma$ is abelian and its action on $X_\rat$ is free.  In the twisted loop algebra setting (where $A = \C[t,t^{-1}]$ and $\Gamma$ is a cyclic group acting on $\g$ by diagram automorphisms), \emph{global} Weyl modules were defined and studied in \cite{FMS11}.  A more general approach, including a definition of global Weyl modules in a large degree of generality, has been taken in~\cite{FMS12}.

\subsection{Moving beyond finite-dimensional Lie algebras as targets} \label{subsec:super-Virasoro}

We conclude this section with a brief overview of some other equivariant map algebras appearing in the literature.  For these, it is necessary to slightly generalize the notion of an evaluation representation.  For later use we state the following definition in the context of Lie superalgebras and we remind the reader that all Lie algebras are also Lie superalgebras.

\begin{defin}[Generalized evaluation representation]
  Let $\g$ be an arbitrary Lie superalgebra.  Suppose $\sfm_1,\dotsc,\sfm_\ell$ are pairwise distinct maximal ideals of $A$, $n_1,\dotsc,n_\ell \in \mathbb{N}_+$, and $\g \otimes (A/\sfm^{n_i}) \to \End_k V_i$ is a representation of $\g \otimes (A/\sfm^{n_i})$ for $i=1,\dots,\ell$.  Then the composition
  \begin{equation} \label{eq:gen-eval} \ts
    \g \otimes A \twoheadrightarrow \bigoplus_{i=1}^\ell (\g \otimes A)/(\g \otimes \sfm^{n_i}) \cong \bigoplus_{i=1}^\ell \g \otimes (A/\sfm^{n_i}) \xrightarrow{\rho} \End_k V_i
  \end{equation}
  is called a \emph{generalized evaluation representation} of $\g \otimes A$.
\end{defin}

Note that if $n_i=1$, then the projection $A \twoheadrightarrow A/\sfm_i$ corresponds to evaluation at the point $x_i \in X$ corresponding to $\sfm_i$.

Instead of the finite-dimensional targets discussed earlier in this section, one can also consider map algebras $\fra \otimes A$, where $\fra$ is an infinite-dimensional Lie algebra.  For example, one can take $\fra$ to be the \emph{Virasoro algebra} $\Vir$.  The unitary irreducible quasifinite $\Vir$-modules (we recall that a module is quasifinite if it has finite-dimensional weight spaces) were classified by Chari and Pressley in \cite{CP88}.  All irreducible quasifinite $\Vir$-modules (without the assumption of unitarity) were then classified by Mathieu in \cite{Mat92}, where it was shown that they are all highest weight modules, lowest weight modules or \emph{modules of the intermediate series} (otherwise known as \emph{tensor density modules}; their nonzero weight spaces are all one-dimensional).  The irreducible quasifinite modules for the \emph{map Virasoro algebra} $\Vir \otimes A$ (where $A$ is finitely generated and $k=\C$) where classified by the second author as follows.
\begin{theo}[{\cite[Th.~5.5]{Sav-Virasoro}}]
  Assume $A$ is a finitely generated $\C$-algebra.  Then any irreducible quasifinite $(\Vir \otimes A)$-module is one of the following:
  \begin{enumerate}
    \item a single point evaluation module corresponding to a $\Vir$-module of the intermediate series,
    \item a generalized evaluation module of the form~\eqref{eq:gen-eval}, where all the $V_i$, $i=1,\dotsc,\ell$, are irreducible highest weight modules, or
    \item a generalized evaluation module of the form~\eqref{eq:gen-eval}, where all the $V_i$, $i=1,\dotsc,\ell$, are irreducible lowest weight modules.
  \end{enumerate}
  In particular, they are all generalized evaluation modules.
\end{theo}
For the twisted case or for other infinite-dimensional $\fra$, very little is known.

One can also consider the case when the target is a \emph{Lie superalgebra}.  In this case, the action of $\Gamma$ on $\g$ is by Lie superalgebra automorphisms (which preserve the even and odd parts of $\g$).  We then call $M(X,\g)^\Gamma$ an \emph{equivariant map (Lie) superalgebra}.  Under the assumptions that $\g$ is a so-called \emph{basic classical Lie superalgebra}, $A$ is finitely generated, $k$ is algebraically closed, and $\Gamma$ acts freely on the rational points of $X$, the irreducible finite-dimensional representations of equivariant map superalgebras were classified by the second author in \cite{Sav-super}. (In the special case of untwisted multiloop superalgebras with basic classical target, they were previously classified in \cite{Rao11,RZ04}.)  The basic classical Lie superalgebras can be split into two types: type I and type II.  A complete list (also indicating the even part $\g_{\bar 0}$ of $\g$) is as follows (see \cite{Kac77}).

\medskip
\begin{center}
  \begin{tabular}{|c|c|c|}
    \hline
    $\g$ & $\g_{\bar 0}$ & Type \\
    \hline
    $A(m,n)$, $m > n \ge 0$ & $A_m \oplus A_n \oplus k$ & I \\
    $A(n,n)$, $n \ge 1$ & $A_n \oplus A_n$ & I \\
    $\mathfrak{sl}(n,n)$, $n \ge 1$ & $A_{n-1} \oplus A_{n-1} \oplus k$ & N/A \\
    $C(n+1)$, $n \ge 1$ & $C_n \oplus k$ & I \\
    $B(m,n)$, $m \ge 0$, $n \ge 1$ & $B_m \oplus C_n$ & II \\
    $D(m,n)$, $m \ge 2$, $n \ge 1$ & $D_m \oplus C_n$ & II \\
    $F(4)$ & $A_1 \oplus B_3$ & II \\
    $G(3)$ & $A_1 \oplus G_2$ & II \\
    $D(2,1;\alpha)$, $\alpha \ne 0,-1$ & $A_1 \oplus A_1 \oplus A_1$ & II \\
    \hline
  \end{tabular}
\end{center}
\medskip

One can define the set $\cE^\Gamma$ as in Section~\ref{subsec:Lie-irr-fd-reps}.  We will use the notation $\cE(\g)^\Gamma$ to denote this set for a target Lie superalgebra $\g$.  We can then define evaluation representations $\ev^\Gamma_\psi$, $\psi \in \cE(\g)^\Gamma$, as in Definition~\ref{def:eval-rep}.  Let $\g_{\bar 0}^\rss = [\g_{\bar 0}, \g_{\bar 0}]$ be the semisimple part of $\g_{\bar 0}$ and let $\g_{\bar 0}^\ab$ be the center of $\g_{\bar 0}$.   Let
\[
  \mathcal{L}(X,\g_{\bar 0}^\ab) = \{\theta \in (\g_{\bar 0}^\ab \otimes A)^* : \theta(\g_{\bar 0}^\ab \otimes I)=0 \text{ for some ideal $I$ of $A$ with finite support}\}.
\]
There is a natural action of $\Gamma$ on $\mathcal{L}(X,\g_{\bar 0}^\ab)$ given by $\gamma \theta := \theta \circ \gamma^{-1}$.  We let $\mathcal{L}(X,\g_{\bar 0}^\ab)^\Gamma$ denote the set of $\Gamma$-invariant elements of $\mathcal{L}(X,\g_{\bar 0}^\ab)$.

\begin{theo}[{\cite[Th.~7.10]{Sav-super}}]
  Suppose $\g$ is a basic classical Lie superalgebra and let $\scS(X,\g)^\Gamma$ be the set of isomorphism classes of irreducible finite-dimensional representations of $(\g \otimes A)^\Gamma$.  Then we have the following:
  \begin{enumerate}
    \item If $\g_{\bar 0}$ is semisimple (i.e.\ $\g$ is of type II or is $A(n,n)$, $n \ge 1$), then the map
        \begin{equation} \label{eq:enumeration-g0-semsimple}
          \cE(\g)^\Gamma \to \scS(X,\g)^\Gamma,\quad \psi \mapsto \ev^\Gamma_\psi,
        \end{equation}
        is a bijection.

    \item If $\g _{\bar 0}$ is of type I, then the map
        \begin{equation} \label{eq:enumeration-g0-not-semsimple}
          \cL(X,\g_{\bar 0}^\ab)^\Gamma \times \cE(\g_{\bar 0}^\rss)^\Gamma \to \scS(X,\g)^\Gamma,\quad (\theta,\psi) \mapsto V^\Gamma(\theta, \psi),
        \end{equation}
        is a bijection, where $V^\Gamma(\theta,\psi)$ is a particular generalized evaluation module, called a \emph{Kac module}, associated to $\theta$ and $\psi$ (see \cite[Def.~7.9]{Sav-super}).
  \end{enumerate}
\end{theo}

%
\section{Open problems} \label{sec:open}
%

In this section we discuss some open problems. The choice of problems reflects of course the authors' taste and is by no means exhaustive.

\subsection{New types of equivariant map algebras} \label{subsec:new-EMAs}

Although we have assumed in Definitions~\ref{Defmap} and~\ref{defEMA} that $\fra$ is finite-dimensional and that $\Ga$ is finite, these restrictions are not necessary for the definitions. There are indeed many natural generalizations. We list here some interesting possibilities.

\begin{asparaenum}
  \item \emph{More general groups $\Gamma$.}  Instead of a finite group $\Gamma$, one can consider, for instance, an algebraic group acting completely reducibly on $A$ and $\fra$.  The first step in the study of the resulting EMAs would be to classify their irreducible finite-dimensional representations.

  \item \emph{More general targets.}  There are many possible choices for the ``target'' algebra $\g$ for which the resulting EMA has not been studied.  Even in the setting of Lie algebras, one could consider $\g$ to be, for instance, a Kac-Moody algebra (see \cite{RF10} for the untwisted multiloop case) or some other well-studied infinite-dimensional Lie algebra, such as the Heisenberg algebra.
      One can also move beyond the Lie (or associative) algebra setting.  In fact, one can study the representation theory of any EMA for which the target is an algebra with well-defined representations.  For example, it would be natural to consider the finite-dimensional representations in the case of alternative algebras or Jordan algebras.

  \item \emph{Superschemes and superalgebras.} The equivariant map superalgebra setting of \cite{Sav-super} can be naturally generalized.  First, one can consider more general Lie superalgebras for the target $\g$.  For example, one could take any finite-dimensional simple Lie superalgebra (not just basic classical ones).  Furthermore, one can replace $A$ by a supercommutative algebra (in other words, let $X$ be a superscheme) and $\Gamma$ by a $\Z_2$-graded group.

  \item \emph{Analytic setting.}  Another interesting problem is to replace the algebraic setting by an analytic one.  For example, let $X$ be a Banach manifold, $\fra$ a Banach algebra and consider equivariant differentiable maps $X \to \fra$.  See, for example, \cite{neeb,neeb-s,wok} for work in this direction.
\end{asparaenum}

\subsection{General equivariant map algebras} \label{subsec:prob-arbit}

Even for the most familiar types of EMAs, there is still much to be learned about their representation theory.  In this subsection, we list some open problems in this direction that apply to
general EMAs (e.g.\ both the associative and Lie setting).  In
Section~\ref{subsec:open-EMLA} we will consider the Lie case in further detail.
In this subsection $\fra$ is an arbitrary algebra of finite dimension over an algebraically closed field. We abbreviate $\frM= M(X,\fra)^\Ga$.

\subsubsection{Isomorphism problem} \label{prob:isomorp}

Characterizing EMAs up to isomorphism is one of the fundamental problems in the theory of EMAs. In the case of loop algebras, a deep result of Kac states that a loop algebra $M(k^\times, \g)^\Ga$ with $\g$ a simple Lie algebra and $\Ga$ cyclic is isomorphic to a loop algebra $M(k^\times, \g)^{\lan \si\ran}$ with $\si$ a diagram automorphism of $\g$ (see \cite[Th.~8.5]{kac}).  Kac' Theorem has recently been extended to the case $M\big( (k^{\times})^2, \frg)^\Ga$ for $\Ga=\lan \si_1, \si_2\ran$ abelian in \cite{abp3}. While this paper uses Lie algebra theory, in the paper \cite{gp} the authors consider $M\big( (k^{\times})^2, \fra)^\Ga$, $\Ga$ as before, for an arbitrary algebra $\fra$ and establish a cohomological invariant for this equivariant map algebra. It is therefore natural, though likely difficult, to try to generalize the results of both papers \cite{abp3} and \cite{gp} to the multiloop case and beyond that to more general EMAs.

\subsubsection{Automorphisms}

Related to the problem of characterizing EMAs up to isomorphism, is the question of determining the automorphism groups of EMAs. In the case where $\g$ is a Lie algebra and $A$ is an integral domain with trivial Picard group, the automorphism group of $\g \ot A$ is described in \cite{pi:aut}; for Lie superalgebras see \cite{GraPi}. Results about the automorphism group of the twisted loop algebra can be obtained by specializing the paper \cite{kac-wang}, which studies the automorphism group of arbitrary Kac-Moody algebras; see also \cite{abp3}. Results about the automorphism groups of multiloop algebras can be found in \cite{gp}.

\subsubsection{Indecomposable modules} \label{prob:indecomposable}

The classification of the irreducible objects  in the category of finite-dimensional modules of $\frM$ for an associative $\fra$ in  Theorem~\ref{th:rep-ass} is only a first step towards a  complete understanding of the category.  Ideally, one would like to describe all of the indecomposable finite-dimensional modules and understand the block decomposition of this category.  See  Section~\ref{prob:Lie-rep} for a more detailed discussion of this problem in the Lie setting.

\subsubsection{Associative bialgebras}

In the case that $\fra$ is an associative bialgebra, one can naturally define multi-orbit evaluation representations using the comultiplication on $\fra$.  It would be interesting to understand these representations.

\subsection{Equivariant map Lie algebras} \label{subsec:open-EMLA}

In this subsection we consider the most well-studied types of EMAs: the equivariant map Lie algebras. We let $\frg$ denote a finite-dimensional Lie algebra over an algebraically closed field of characteristic zero and we use the acronym EMA to mean ``equivariant map Lie algebra''.  As usual, $\frM=M(X,\g)^\Ga$.

\subsubsection{Representation Theory} \label{prob:Lie-rep}

As mentioned in Section~\ref{subsec:prob-arbit}, the study of the representation theory of EMAs goes well beyond the classification of irreducible finite-dimensional modules.  Even in the Lie setting, where the extensions between irreducible finite-dimensional modules have been described in \cite{NS}, there are many open problems.  We list here some particularly important ones.

\begin{asparaenum}
  \item \emph{Higher extensions.} Only the first Ext groups are calculated in \cite{NS}, where it is shown that these groups are related to both the representation theory of the target $\g$ and the geometry of the scheme $X$ (see, for instance, Proposition~\ref{prop:ext-free-abelian}, which involves the tangent space).  It is natural to try to compute higher Ext groups.  One would expect a rich interplay between these groups and the geometry of the scheme $X$.  It would be particularly interesting to see how the global dimension of the category of finite-dimensional modules is related to this geometry.

  \item \label{open-prob:Weyl-modules} \emph{Weyl modules.}  The study of (global and local) Weyl modules for EMAs is an active area of research.  In particular, one would like to extend the results on the lengths, dimensions and characters of the local Weyl modules (see \cite{CP01,BN04,CL06,CM04,FoL07,Nak01,Nao10}) to the setting of more arbitrary EMAs (see~\cite{FMS12} for the definitions of global and local Weyl modules in this setting).  It is also natural to try to extend the notion of BGG reciprocity described in \cite{BCM11,BBCKL12} to a more general EMA setting (see \cite[Rem.~2.7]{BBCKL12}).

  \item \emph{Possibly infinite-dimensional representations.}  Much of the current focus in the study of the representation theory of arbitrary EMAs has been on the category of finite-dimensional representations.  However, in specific cases (most importantly, the loop and current algebras) more general categories of representations have been studied.  For instance, one can consider a type of category $\mathcal{O}$ for (central extensions of) loop algebras (see, for example, the survey article \cite{Cha11}).  It would be interesting to study analogous categories in the more arbitrary EMA setting.  For example, one could consider the category of integrable highest weight representations, the category of quasifinite representations (i.e.\ representations with finite-dimensional weight spaces) or appropriate generalizations of category $\mathcal{O}$.
\end{asparaenum}

\subsubsection{Cohomology} \label{prob:Lie-coho}

Assume $\g$ is semisimple.  The formulas for $\Ext_\frM(V_1, V_2)$ proven in \cite{NS} have a cohomological interpretation since
\[
  \Ext_\frM(V_1, V_2) \cong \rmH^1\big(\frM, \Hom_k(V_1, V_2)\big).
\]
One can therefore expect that it should be possible to describe the first cohomology group $\rmH^1(\frM, V)$ for a general $\frM$-module $V$. For some equivariant map algebras $\frM$, including multiloop algebras, and locally-finite $\frM$-modules, this is done in \cite{NP}. A particular instance of this question is the adjoint module $\frM_{\ad}$, for which $\rmH^1(\frM, \frM_\mathrm{ad}) = \Der_k(\frM)/\IDer(\frM)$. The structure of the derivation algebra in the untwisted case $\frM=\g\ot A$ is determined in \cite{block,bemo,Azam:tensor} and for multiloop algebras in \cite{Azam:multi,pi:der}. In both cases, $\g$ need not be a Lie algebra but can be an arbitrary finite-dimensional $k$-algebra.

A related problem is to determine the homology groups $\rmH_*(\frM, V)$. Some
results on $\rmH_0(\frM, \frM_\mathrm{ad}) = \frM/[\frM,\frM]$ are contained in
\cite{NSS} and \cite{NS}, see for example Lemmas~6.1 and 6.10 of \cite{NS}
describing classes of perfect equivariant map algebras. The interest in
perfect equivariant map algebras comes from Theorem~\ref{thm:Lie} which
indicates that these algebras have a simpler representation
theory. A second reason is that perfect Lie algebras have universal central
extensions, as discussed in Section~\ref{prob:uce} below.

\subsubsection{Central extensions}\label{prob:uce}

Central extensions of Lie algebras often have a simpler representation
theory, the affine Kac-Moody Lie algebras being a case in point. It is
therefore of interest to understand central extensions of equivariant map
algebras $\frM$, preferably the universal central extension $\uce(\frM)$ of a perfect $\frM$.

In the untwisted case $\uce(\g \ot A)$ is described in \cite{Kas} and later
again in \cite{Rao-Moody}. In general, the canonical map $\uce(\g\ot A) \to \g
\ot A$ has an infinite-dimensional kernel, contrary to what happens in
the case of loop algebras. For example, this is already so for $A=k[t_1^{\pm
1}, \ldots, t_n^{\pm 1}]$ and $n\ge 2$.

Generalizing the example of twisted loop algebras, one would like to relate the universal central extension of a perfect equivariant map algebra $(\g\ot A)^\Ga$ to the well-understood universal central extension $\uce(\g\ot A)$. Observe that the action of $\Ga$ on $\g \ot A$ uniquely lifts to an action on $\uce(\g \ot A)$. The formula one is aiming for is
then very natural: $\uce\big( (\g \ot A)^\Ga\big) \cong \big(\uce(\g \ot A)\big)^\Ga$. It is known that this is indeed true for certain multiloop algebras, the so-called Lie tori (see~\cite{n:eala,sun:uce}).


\bibliographystyle{alpha}
\bibliography{neher-savage-survey-biblist}

\end{document}